\documentclass[12pt,reqno]{article}

\usepackage[usenames]{color}
\usepackage{amssymb}
\usepackage{graphicx}
\usepackage{amscd}

\usepackage[colorlinks=true,
linkcolor=webgreen,
filecolor=webbrown,
citecolor=webgreen]{hyperref}

\definecolor{webgreen}{rgb}{0,.5,0}
\definecolor{webbrown}{rgb}{.6,0,0}

\usepackage{color}
\usepackage{fullpage}
\usepackage{float}

\usepackage{graphics,amsmath,amssymb}
\usepackage{amsthm}
\usepackage{amsfonts}
\usepackage{latexsym}
\usepackage{epsf}

\usepackage{multirow}
\usepackage{multicol}
\usepackage{scalerel}
\usepackage{arydshln}

\def\msquare{\mathord{\scalerel*{\Box}{gX}}}
\def\modd#1 #2{#1\ \mbox{\rm (mod}\ #2\mbox{\rm )}}
\def\tmodd#1 #2{\mbox{\tiny{$\modd{#1} {#2}$}}}

\begin{document}

%\begin{center}
%\epsfxsize=4in
%\leavevmode\epsffile{logo129.eps}
%\end{center}

\theoremstyle{plain}
\newtheorem{theorem}{Theorem}
\newtheorem{corollary}[theorem]{Corollary}
\newtheorem{lemma}[theorem]{Lemma}
\newtheorem{proposition}[theorem]{Proposition}

\theoremstyle{definition}
\newtheorem{definition}[theorem]{Definition}
\newtheorem{example}[theorem]{Example}
\newtheorem{conjecture}[theorem]{Conjecture}

\theoremstyle{remark}
\newtheorem{remark}[theorem]{Remark}

\begin{center}

\vskip 1cm

{\LARGE\bf Curious Bounds for Floor Function Sums}

\vskip 1cm

\large
Thotsaporn Thanatipanonda and Elaine Wong\footnote{Correspondence should be addressed to Elaine Wong: \href{mailto:wongey@gmail.com}{\tt wongey@gmail.com} .} \\ Science Division \\ Mahidol University International College \\ Nakhon Pathom, 73170 \\ Thailand \\
\href{mailto:thotsaporn@gmail.com}{\tt thotsaporn@gmail.com} \\
\href{mailto:wongey@gmail.com}{\tt wongey@gmail.com}

\end{center}

\vskip .2 in

\begin{abstract}
The sums of floor functions have been studied by Jacobsthal, Carlitz, Grimson, and Tverberg. More recently, Onphaeng and Pongsriiam proved some sharp upper and lower bounds for the sums of Jacobsthal and Tverberg.  In this paper, we devise concise formulas for the sums and then use it to give proofs of the upper and lower bounds that were claimed by Tverberg. Furthermore, we present conjectural lower and upper bounds for these sums.
\end{abstract}

\section{Introduction}
In 1957, Jacobsthal \cite{jacobsthal} defined and studied a function of the form $$f_m(\{ a_1, a_2 \},k)=\left\lfloor\frac{a_1+a_2
+k}{m}\right\rfloor-\left\lfloor\frac{a_1+k}{m}\right\rfloor-\left\lfloor\frac{a_2+k}{m}\right\rfloor+\left\lfloor\frac{k}{m}\right\rfloor$$ for fixed $m\in\mathbb{Z}^+$ with $a_1,a_2,k \in\mathbb{Z}$.  He also defined the functions $$S_m(\{ a_1,a_2 \}, K)=\sum\limits_{k=0}^Kf_m(\{ a_1,a_2\},k), \qquad 0\leq a_1,a_2,K\leq m-1.$$  It is of interest to note that we will take advantage of the $m$-periodicity of $f_m$, and so we restrict our $a_1,a_2,$ and $K$ values accordingly for the sum. Jacobsthal, then later Carlitz \cite{carlitz}, Grimson \cite{grimson}, and Tverberg \cite{tverberg} proved $S_m(\{ a_1, a_2 \},K)\geq 0$.  In 2012, Tverberg \cite{tverberg} proposed a generalized notation for these sum functions for any set $A =\{ a_1,\ldots,a_n \}$ with $0\leq a_1,\ldots,a_n,K\leq m-1$ and $n=|A|$, that is, $$S_m(\{ a_1,\ldots,a_n\} ,K)=\sum\limits_{k=0}^K \sum\limits_{T\subset[1,n]}(-1)^{n-|T|}\left\lfloor\frac{k+\sum_{i\in T} a_i}{m}\right\rfloor.$$ He also claimed without proof the other upper and lower bounds of $S_m$ for sets $\{ a_1, a_2 \} $ and $\{ a_1,a_2,a_3\}$ (i.e., $n=2,3$).  In 2017, Onphaeng and Pongsriiam \cite{day} furnished a proof for the upper bounds when $n$ is even and $\geq 4$ and the lower bounds when $n$ is odd and $\geq 3$.  In this paper, we investigate the bounds for $S_m$ for all $n\in\mathbb{Z}^+$ and supply the missing proofs of Tverberg's upper bounds.  Furthermore, we conjecture all the bounds for $S_m$ not previously mentioned and summarize the findings in Table \ref{table:credit}. Credit attributed to authors who claim the statement without proof are denoted with an asterisk (*). Otherwise, a proof is given in their corresponding paper.

\begin{center}
\begin{table}[H]
{\renewcommand{\arraystretch}{1.2}
\resizebox{\columnwidth}{!}{%
\begin{tabular}{|c||c|c||c|c|}
\hline
$n$& Lower Bound & Lower Bound Credit & Upper Bound & Upper Bound Credit\\
\hline
1 & 0 & Trivial & $m-1$ & Trivial \\
\hline
\multirow{5}{*}{2} & \multirow{5}{*}{0} & Jacobsthal; & \multirow{5}{*}{$\left\lfloor\frac{m}{2}\right\rfloor$} & \multirow{5}{*}{\begin{tabular}{c}Tverberg*;\\ \textbf{Corollary \ref{thm:2varfull}} \end{tabular}}\\
&&Carlitz;&&\\
&&Grimson;&&\\
&&Tverberg; &&\\
&& \textbf{Theorem \ref{thm:2varlowerbound}} &&\\
\hline
\multirow{3}{*}{3} & \multirow{3}{*}{$-2\left\lfloor\frac{m}{2}\right\rfloor$} &Tverberg*; & \multirow{3}{*}{$\left\lfloor\frac{m}{3}\right\rfloor$} & \multirow{3}{*}{\begin{tabular}{c} Tverberg*;\\ \textbf{Corollary \ref{thm:3varfull}} \end{tabular}}\\
&& Onphaeng, && \\
&& Pongsriiam && \\
\hline
4 & $-3\left\lfloor\frac{m}{3}\right\rfloor$ & \multirow{2}{*}{\textbf{Conjecture \ref{conj:main}*}} &$4\left\lfloor\frac{m}{2}\right\rfloor$ & Onphaeng,\\
&(Conjecture)&&&Pongsriiam\\
\hline
odd &\multirow{2}{*}{$-2^{n-2}\left\lfloor\frac{m}{2}\right\rfloor$}&Onphaeng, & \multirow{2}{*}{(Conjectures)} & \multirow{2}{*}{\textbf{Conjecture \ref{conj:main}*}}\\
($\geq 5$)&&Pongsriiam &&\\
\hline
even&\multirow{2}{*}{(Conjectures)}&\multirow{2}{*}{\textbf{Conjecture \ref{conj:main}*}}&\multirow{2}{*}{$2^{n-2}\left\lfloor\frac{m}{2}\right\rfloor$}&Onphaeng,\\
($\geq 5$)&&&&Pongsriiam\\
\hline
\end{tabular}%
}
}
\label{table:credit}
\end{table}
\end{center}

We discuss the case of $n=1$ in this section as it sets the foundation for the main strategy that we use to prove higher cases. We begin with an explicit definition of Jacobsthal's sum.

\begin{definition} For any $m\in\mathbb{Z}$ and $a_1,K\in\mathbb{Z}^+\cup\{0\}$,
\begin{equation*}
S_m(\{ a_1\} ,K)=\sum\limits_{k=0}^K \left(\left\lfloor\frac{a_1+k}{m}\right\rfloor - \left\lfloor\frac{k}{m}\right\rfloor\right).
\end{equation*}
\end{definition}

The sum can be written concisely without using the summation symbol, which we will show below. Note that the periodicity that existed in the $n\geq 2$ case does not exist here.  However, we only prove the following proposition for $0\leq K\leq m-1$ because that is all that will be needed for higher values of $n$. From this point on, we let $a$ $(\mbox{mod } m)$ denote the minimal non-negative representative in the $\mathbb{Z}/m$-equivalence class.

\begin{proposition} For $0\leq K \leq m-1$ and $a_1\in\mathbb{Z}^+\cup\{0\}$,
\label{prop:main1}
\begin{equation*}
S_m(\{ a_1\} ,K)=\left\lfloor\frac{a_1}{m}\right\rfloor (K+1) + \max (0, \modd{a_1} {m} + K - m + 1).
\end{equation*}
\end{proposition}

\begin{proof}

We observe $$\left\lfloor\frac{a_1+k}{m}\right\rfloor = \left\lfloor\frac{a_1}{m}\right\rfloor + \left\lfloor\frac{k}{m}\right\rfloor + \left\lfloor\left\lbrace\frac{a_1}{m}\right\rbrace + \left\lbrace\frac{k}{m}\right\rbrace\right\rfloor,$$ where $\{ x \}=x - \left\lfloor x\right\rfloor$, the fractional part of $x$. This notation is distinguished from the usual set notation according to context. Furthermore, since $0\leq k \leq m-1$, $$\left\lfloor\left\lbrace\frac{a_1}{m}\right\rbrace + \left\lbrace\frac{k}{m}\right\rbrace\right\rfloor = \left\lfloor \frac{\modd{a_1} {m}}{m}+\frac{k}{m}\right\rfloor.$$ The result is derived as follows.

\begin{align*}
\sum\limits_{k=0}^K \left(\left\lfloor\dfrac{a_1+k}{m}\right\rfloor\right.& - \left. \left\lfloor\dfrac{k}{m}\right\rfloor\right)\\
= & \sum\limits_{k=0}^K\left(\left\lfloor\frac{a_1}{m}\right\rfloor + \left\lfloor\frac{k}{m}\right\rfloor+\left\lfloor \left\lbrace\frac{a_1}{m}\right\rbrace + \left\lbrace\frac{k}{m}\right\rbrace\right\rfloor - \left\lfloor\frac{k}{m}\right\rfloor \right) \\
= &\left\lfloor\frac{a_1}{m}\right\rfloor(K+1) + \sum_{k=0}^K\left\lfloor\frac{\modd{a_1} {m}}{m}+\frac{k}{m}\right\rfloor\\
= & \left\lfloor\frac{a_1}{m}\right\rfloor(K+1)+ \sum\limits_{k=0}^{m-(\tmodd{a_1} {m})-1} \left\lfloor\frac{\modd{a_1} {m} +k}{m}\right\rfloor\\
& +\sum\limits_{k=m-\tmodd{a_1} {m}}^K \left\lfloor\frac{\modd{a_1} {m}+k}{m}\right\rfloor \\
= & \left\lfloor\frac{a_1}{m}\right\rfloor(K+1)+\sum\limits_{k=0}^{m-(\tmodd{a_1} {m})-1} 0 +\sum\limits_{k=m-\tmodd{a_1} {m}}^K 1 \\
= & \left\lfloor\frac{a_1}{m}\right\rfloor (K+1) + \max (0,\modd{a_1} {m} + K - m + 1)
\end{align*}

\end{proof}

The bounds for $S_m(\{a_1\},K)$ are then easily attained from Proposition \ref{prop:main1}.

\begin{corollary}
For $0\leq a_1,K \leq m-1$, $$0\leq S_m(\{ a_1 \} ,K) \leq m-1 .$$ In particular, the maximum occurs precisely when $a_1=K=m-1$.
\end{corollary}

\begin{proof}
The result follows from the fact that $$S_m(\{ a_1\},K)=\max (0,a_1+K-m+1).$$ 
\end{proof}

In the following sections, we use similar methods to provide bounds when $n>1$ for the sums $S_m(\{ a_1,\ldots,a_n \},K)$.

\section{Lower and Upper Bounds for \texorpdfstring{$n=2$}{}}

The lower bound for $n=2$ has been shown by Carlitz \cite{carlitz}, Grimson \cite{grimson}, Jacobsthal \cite{jacobsthal}, and Tverberg \cite{tverberg}, while the upper bound was first mentioned by Tverberg.  Like in the case of $n=1$, we introduce a new form for the sum by generalizing Proposition \ref{prop:main1} and then prove its upper bound.  We also use this form to provide a new proof for its lower bound. We begin like before, by writing out the sum of Jacobsthal explicitly.

\begin{definition}
\label{def:2varsum}
For any $m\in\mathbb{Z}^+$ and any $a_1, a_2, K \in \mathbb{Z}^+ \cup \{0\}$,
\begin{equation*}
S_m(\{  a_1,a_2\},K)=\sum\limits_{k=0}^K \left(\left\lfloor\frac{a_1+a_2+k}{m}\right\rfloor - \left\lfloor\frac{a_1+k}{m}\right\rfloor - \left\lfloor\frac{a_2+k}{m}\right\rfloor + \left\lfloor\frac{k}{m}\right\rfloor\right).
\end{equation*}
\end{definition}

We show that the sum can be written concisely without using the summation symbol, similar to Proposition \ref{prop:main1}.

\begin{proposition}
\label{prop:main2}
For $0\leq K \leq m-1$, and any $a_1,a_2 \in \mathbb{Z}^+\cup \{0\}$,
\begin{align*}
S_m(\{ a_1,a_2\}, K) =& \left(\left\lfloor\frac{a_1+a_2}{m}\right\rfloor-\left\lfloor\frac{a_1}{m}\right\rfloor- \left\lfloor\frac{a_2}{m}\right\rfloor \right)(K+1)\\
& +\max (0, \modd{(a_1+a_2)} {m} + K-m+1)\\
& -\max (0, \modd{a_1} {m} + K-m+1)\\
&-\max (0, \modd{a_2} {m} + K-m+1).
\end{align*}
\end{proposition}

\begin{proof}
Rewriting the two-variable sum in Definition \ref{def:2varsum} as a series of one-variable sums,  $$S_m(\{ a_1, a_2 \}, K)=S_m(\{ a_1+a_2 \},K)-S_m(\{ a_1 \},K)-S_m(\{ a_2 \}, K),$$ allows us to apply Proposition \ref{prop:main1} to each sum to get our result.
\end{proof}

A symmetry exists in Proposition \ref{prop:main2}. We outline the pattern in the lemma below.  This, along with a partial result in Theorem \ref{thm:2var}, gives us the desired upper and lower bounds in Corollary \ref{thm:2varfull} and Theorem \ref{thm:2varlowerbound}, respectively.

\begin{lemma}
\label{thm:2reflect}
(Mirrored Sums) For $0\leq a_1,a_2\leq m-1$ and $0\leq K\leq m-2$, $$S_m(\{ a_1,a_2 \},K) = S_m(\{ m-a_1,m-a_2 \},m-2-K).$$
\end{lemma}

\begin{proof}
It is enough to show the claim for $0\leq a_1+a_2 \leq m$. Otherwise, we have that $(m-a_1)+(m-a_2) < m$, in which case we can use a similar argument by substituting $a_1$ with $m-a_1$ and $a_2$ with $m-a_2$.

For the case $a_1=a_2 = 0$, the result trivially holds by Definition \ref{def:2varsum}. For the case $0<a_1+a_2<m$, Proposition \ref{prop:main2} simplifies to

\begin{equation}
\begin{split}
\label{eq:reflectleft}
S_m(\{ a_1&,a_2\}, K)=\max (0,a_1+a_2+K-m+1)\\ 
&- \max (0,a_1+K-m+1) - \max (0,a_2+K-m+1).
\end{split}
\end{equation} Furthermore, we note that $m<2m-(a_1+a_2)<2m$, which gives

\begin{align}
\begin{split}
\label{eq:reflectright}
S_m(\{ m-a_1&,m-a_2\}, m-2-K)= m-1-K+\max (0,m-(a_1+a_2)-K-1)\\
&- \max (0,m-a_1-K-1) - \max (0,m-a_2-K-1).
\end{split}
\end{align} Now consider the following equations that use the fact that $$\max (0,x)-\max (0,-x)=x, \mbox{ for all }x \in \mathbb{R}.$$ \begin{align}
\label{eq:max1} \begin{split} \max &(0,a_1+a_2+K-m+1) -\max (0,-a_1-a_2-K+m-1) \\
& = a_1+a_2 + K-m+1, \end{split} \\
\label{eq:max2} \begin{split} \max &(0,m-a_1-K-1)-\max (0,a_1+K-m+1)\\ & = m-a_1-K-1, \end{split} \\
\label{eq:max3} \begin{split} \max &(0,m-a_1-K-1)- \max (0,a_2+K-m+1)\\
& = m-a_2-K-1. \end{split}
\end{align} Then, (\ref{eq:max1})+(\ref{eq:max2})+(\ref{eq:max3}) confirms (\ref{eq:reflectleft}) = (\ref{eq:reflectright}). Finally, we consider $a_1+a_2=m$. Here, Proposition \ref{prop:main2} gives \begin{align}
\begin{split}
\label{eq:reflectmleft} S_m & (\{ a_1, a_2 \}, K)\\
& = K +1- \max (0,a_1+K-m+1)-\max (0,a_2+K-m+1),
\end{split}\\
\begin{split} 
\label{eq:reflectmright} S_m&(\{ m-a_1, m-a_2 \}, m-2-K)\\
& = m-(K+1) - \max (0,m-a_1-K-1) - \max (0,m-a_2-K-1).
\end{split}
\end{align} In this case, (\ref{eq:max2})+(\ref{eq:max3}) confirms (\ref{eq:reflectmleft}) = (\ref{eq:reflectmright}). This concludes the proof.

\end{proof}

We show the upper bound for half the range of $K$ using differences.

\begin{theorem}
\label{thm:2var}
For $0\leq a_1,a_2 \leq m-1$ and $0 \leq K \leq \left\lfloor \frac{m}{2}\right\rfloor - 1$,
\begin{equation*}
S_m(\{ a_1,a_2 \}, K) \leq \left\lfloor\frac{m}{2}\right\rfloor.
\end{equation*}
\end{theorem}

\begin{proof}
We show the stronger result, that for $0\leq K\leq \left\lfloor\frac{m}{2}\right\rfloor -1$, $$S_m(\{ a_1,a_2\},K)\leq K+1.$$

For the case $K=0$,
\begin{align*}
S_m(\{ a_1, a_2 \}, 0) = & \left\lfloor\frac{a_1+a_2}{m}\right\rfloor + \max (0,\modd{(a_1+a_2)} {m}-m+1) \\
&-\max (0,a_1-m+1)-\max (0,a_2-m+1) \\
= & \left\lfloor\frac{a_1+a_2}{m}\right\rfloor + 0 - 0 - 0 \\
\leq &1.
\end{align*} For the case $1 \leq K\leq\left\lfloor\frac{m}{2}\right\rfloor - 1$, it is enough to show that $$\Delta_m:=S_m(\{ a_1,a_2\},K)-S_m(\{ a_1+1,a_2 \},K-1)\leq 1.$$ By using Proposition \ref{prop:main2} we explicitly write out the two sums \begin{align*}
S_m(\{ a_1,a_2\}, K)=&\left\lfloor\frac{a_1+a_2}{m}\right\rfloor (K+1)\\
& +\max (0,\modd{(a_1+a_2)} {m} + K-m+1)\\
& -\max (0,a_1+K-m+1)\\
& -\max (0,a_2+K-m+1),\\
S_m(\{ a_1+1,a_2\}, K-1) =& \left\lfloor\frac{a_1+a_2+1}{m}\right\rfloor K\\
& +\max (0,\modd{(a_1+a_2+1)} {m} + K - m)\\
& -\max (0,a_1+K-m+1)\\
& -\max (0,a_2+K-m).
\end{align*} We determine $\Delta_m$ in four cases according to the possible values of $a_1+a_2$ and $a_2+K-m+1$.

\textbf{Case 1: ($a_1+a_2<m$ and $a_2+K-m+1 \leq 0$)}

If $a_1+a_2<m-1$, then both sums are the same. \begin{align*}
S_m(\{ a_1,a_2\}, K) =  0 & + \max (0,a_1+a_2 + K- m+1)\\ 
&-\max (0,a_1+K-m+1) - 0,\\
S_m(\{ a_1+1,a_2\}, K-1)  = 0 & + \max (0,(a_1+1)+a_2 + K - m)\\ 
&-\max (0,(a_1+1)+K-m) - 0.
\end{align*} If $a_1+a_2=m-1$, then both sums evaluate to $K$.  Therefore, we get $\Delta_m=0$ for this case.

\textbf{Case 2: ($a_1+a_2<m$ and $a_2+K-m+1 > 0$)}

Again, we assume $a_1+a_2<m-1$.
\begin{align*}
S_m(\{ a_1,a_2\}, K) = 0 & + \max (0,a_1+a_2 + K - m+1)\\ 
&-\max (0,a_1+K-m+1) - (a_2+K-m+1),\\
S_m(\{ a_1+1,a_2\}, K-1) = 0 & + \max (0,(a_1+1)+a_2 + K - m+1)\\ 
&-\max (0,(a_1+1)+K-m) - (a_2+K-m).
\end{align*} Therefore, $\Delta_m=S_m(\{ a_1,a_2\}, K)- S_m(\{ a_1+1,a_2\}, K-1)=-1.$ A similar argument holds for $a_1+a_2=m-1$.

\textbf{Case 3: ($a_1+a_2\geq m$ and $a_2+K-m+1 \leq 0$)} 

\begin{align*}
S_m(\{ a_1,a_2\}, K)=(K & +1)+\max (0,(a_1+a_2-m)+K-m+1)\\
&-\max (0,a_1+K-m+1) - 0\\
S_m(\{ a_1+1,a_2\}, K-1) = K & +\max (0,(a_1+1+a_2-m)+K-m)\\
&-\max (0,(a_1+1)+K-m) - 0
\end{align*} Therefore, $\Delta_m=S_m(\{ a_1,a_2\}, K) -S_m(\{ a_1+1,a_2\}, K-1) = +1$.

\textbf{Case 4: ($a_1+a_2 \geq m$ and $a_2+K-m+1 >0$)}

$\Delta_m=0$ using similar reasoning to Case 2 and Case 3. 

We summarize the four cases for $0<K\leq\left\lfloor\frac{m}{2}\right\rfloor-1$ in the table below.
\begin{table}[ht]
{\renewcommand{\arraystretch}{1.2}
\begin{center}
\begin{tabular}{|c|c|c|c|}
\hline
Case & $a_1+a_2$ & $a_2+K-m+1$ & $\Delta_m$\\
\hline
1 & $<m$ & $\leq 0$ & 0\\
2 & $<m$ & $>0$ & $-1$\\
3 & $\geq m$ & $\leq 0$ & $+1$\\
4 & $\geq m$ & $>0$ & 0\\
\hline
\end{tabular}
\end{center}
}
\caption{Summary of $\Delta_m$ values.}
\label{table:chartdelta}
\end{table}

This shows that $-1\leq \Delta_m \leq 1$. Thus, we have shown that $$S_m(\{ a_1,a_2 \}, K) \leq S_m(\{ a_1,a_2\}, K-1) + 1\leq  K+1 \leq \left\lfloor\frac{m}{2}\right\rfloor$$ for all $0\leq K\leq \left\lfloor\frac{m}{2}\right\rfloor-1$.
\end{proof}

We apply Lemma \ref{thm:2reflect} and Theorem \ref{thm:2var} to show the upper bound for $S_m$.

\begin{corollary}
\label{thm:2varfull}
For $0\leq a_1,a_2,K\leq m-1$,$$S_m(\{ a_1,a_2 \}, K) \leq \left\lfloor\frac{m}{2}\right\rfloor.$$
\end{corollary}
\begin{proof}
For $0\leq K \leq \left\lfloor \frac{m}{2}\right\rfloor -1$, Theorem \ref{thm:2var} gives the result. For $\left\lfloor\frac{m}{2}\right\rfloor \leq K \leq m-2$, $S_m(\{ a_1,a_2\}, K) = S_m(\{ m-a_1,m-a_2 \}, m-2-K)$ by Lemma \ref{thm:2reflect}.  From there, we apply Theorem \ref{thm:2var} to the right hand side and get the result. Finally, for $K=m-1$, it is easily seen that $S_m(\{ a_1,a_2 \}, K)=0$. This completes the proof.
\end{proof}

The lower bound is now easy to show using $\Delta_m$.

\begin{theorem}
\label{thm:2varlowerbound}
For $0\leq a_1,a_2,K\leq m-1$, $$0 \leq S_m(\{ a_1,a_2 \}, K).$$
\end{theorem}

\begin{proof}
Without loss of generality, we assume that $0\leq a_2 \leq a_1 \leq m-1$.  Consider $0\leq K \leq \left\lfloor \frac{m}{2}\right\rfloor-1$.  Thus, the conditions of Case 2 (i.e. $a_1+a_2 < m$ and $a_2+K-m+1>0$) cannot be met because if $a_1+a_2< m$, then $a_2 \leq \left\lfloor\frac{m}{2} \right\rfloor$.  So, $a_2+K-m+1 \leq 0$.  This shows that  $\Delta_m \neq -1$, which means $\Delta_m=0$ or 1.  This, along with $S_m(\{ a_1,a_2 \}, 0) \geq 0$, gives us the lower bound as desired.  Next, we consider $\left\lfloor\frac{m}{2}\right\rfloor \leq K\leq m-2$, and apply Lemma \ref{thm:2reflect} to complete the argument. Lastly, $S_m(\{ a_1,a_2 \}, m-1)=0$.  We now have the complete result.
\end{proof}

\section{Upper Bound for \texorpdfstring{$n=3$}{}}

In this section, we follow the previous style of rewriting the sum, observing its symmetry, and using a difference to prove the upper bound. The lower bound has already been  proven in \cite{day}. This time, we use Tverberg's formulation to write out the sum explicitly.

\begin{definition}
\label{def:3varsum}
For any $m\in\mathbb{Z}^+$ and $a_1, a_2, a_3, K\in\mathbb{Z}^+\cup\{0\}$,
\begin{align*}
S_m(\{ a_1,a_2,a_3\},K)=&\sum\limits_{k=0}^K \left(\left\lfloor\frac{a_1+a_2+a_3+k}{m}\right\rfloor\right. \\
& - \left\lfloor\frac{a_1+a_2+k}{m}\right\rfloor-\left\lfloor\frac{a_2+a_3+k}{m}\right\rfloor-\left\lfloor\frac{a_1+a_3+k}{m}\right\rfloor\\
& + \left. \left\lfloor\frac{a_1+k}{m}\right\rfloor+\left\lfloor\frac{a_2+k}{m}\right\rfloor+\left\lfloor\frac{a_3+k}{m}\right\rfloor-\left\lfloor\frac{k}{m}\right\rfloor\right).
\end{align*}
\end{definition}

Like before, we show that the sum can be written concisely without using the summation symbol.

\begin{proposition}
\label{prop:main3} 
For $0\leq K \leq m-1$ and $a_1, a_2, a_3\in\mathbb{Z}^+\cup\{0\}$,
\begin{align*}
S_m(\{ a_1,a_2,a_3\}, K)=&\left(\left\lfloor\frac{a_1+a_2+a_3}{m}\right\rfloor - \left\lfloor\frac{a_1+a_2}{m}\right\rfloor - \left\lfloor\frac{a_2+a_3}{m}\right\rfloor\right.\\
& \left. -\left\lfloor\frac{a_1+a_3}{m}\right\rfloor + \left\lfloor\frac{a_1}{m}\right\rfloor + \left\lfloor\frac{a_2}{m}\right\rfloor + \left\lfloor\frac{a_3}{m}\right\rfloor\right)(K+1)\\
& + \max (0,\modd{(a_1+a_2+a_3)} {m} + K - m + 1)\\
& - \max (0,\modd{(a_1+a_2)} {m} + K - m + 1)\\
& - \max (0,\modd{(a_2+a_3)} {m} + K - m + 1)\\
& - \max (0,\modd{(a_1+a_3)} {m} + K - m + 1)\\
& + \max (0,\modd{a_1} {m}+K-m+1)\\ 
& + \max (0,\modd{a_2} {m}+K-m+1) \\
& + \max (0,\modd{a_3} {m}+K-m+1).
\end{align*}
\end{proposition}
\begin{proof}

We can rewrite our three-variable sum in Definition \ref{def:3varsum} as the two-variable sums

\begin{equation}
\label{eq:threetotwo}
S_m(\{ a_1, a_2, a_3 \}, K) = S_m(\{ a_1,a_2+a_3 \},K)-S_m(\{ a_1,a_2 \},K)-S_m(\{ a_1,a_3 \}, K).
\end{equation}

We then apply Proposition \ref{prop:main2} to each sum to get our result.

\end{proof}

A symmetry exists for $S_m (\{ a_1,a_2,a_3 \}, K)$, similar to Lemma \ref{thm:2reflect}.

\begin{lemma}
\label{thm:3reflect}
(Mirrored Sums) For $0\leq a_1,a_2,a_3\leq m-1$ and $0\leq K\leq m-2$, $$S_m(\{ a_1,a_2,a_3 \},K) = S_m(\{ m-a_1,m-a_2,m-a_3 \},m-2-K).$$
\end{lemma}

\begin{proof}
The three-variable sum can be rewritten as a series of two-variable sums and we can reason as follows:
\begin{align*}
S_m& (\{ a_1,a_2,a_3 \}, K)\\
= & S_m(\{ a_1,a_2+a_3 \},K) - S_m(\{ a_1,a_2 \},K)-S_m(\{ a_1,a_3 \}, K)\\
= & S_m(\{ a_1,\modd{(a_2+a_3)} {m} \},K)-S_m(\{ a_1,a_2 \},K)-S_m(\{ a_1,a_3 \}, K)\\
 = & S_m(\{ m-a_1, m-\modd{(a_2+a_3)} {m} \}, m-2-K) \\
& - S_m(\{ m-a_1, m-a_2 \}, m-2-K) \\
& - S_m(\{ m-a_1, m-a_3 \}, m-2-K) \\
= & S_m(\{ m-a_1, (m-a_2) + (m-a_3) \}, m-2-K) \\ 
& - S_m(\{ m-a_1, m-a_2 \}, m-2-K) \\
& - S_m(\{ m-a_1, m-a_3 \}, m-2-K) \\
= & S_m(\{ m-a_1, m-a_2, m-a_3 \}, m-2-K).
\end{align*}

The first and fifth equalities come from (\ref{eq:threetotwo}). We take advantage of the $a_i$-periodicity of the sums in the second and fourth equalities. Lastly, we apply Lemma \ref{thm:2reflect} in the third equality.
\end{proof}

Next, we show the upper bound for half the range of $K$ using $\Delta_m$.

\begin{theorem}
\label{thm:3var}
For $0\leq a_1,a_2,a_3 \leq m-1$ and $0 \leq K \leq \left\lfloor \frac{m}{2}\right\rfloor - 1$, $$S_m(\{ a_1,a_2,a_3 \}, K) \leq \left\lfloor\frac{m}{3}\right\rfloor.$$
\end{theorem}

\begin{proof}
Without loss of generality, we assume that $0\leq a_3 \leq a_2 \leq a_1 \leq m-1$. We break down the proof into two cases of $K$, that is, we would like to show

$$ S_m(\{ a_1,a_2,a_3 \}, K) \leq \begin{cases}
K+1 & \mbox{ if } 0\leq K \leq \left\lfloor\frac{m}{3}\right\rfloor-1 \mbox{ (\textbf{Case A});} \\
\left\lfloor\frac{m}{3}\right\rfloor & \mbox{ if } \left\lfloor\frac{m}{3}\right\rfloor \leq K \leq \left\lfloor\frac{m}{2}\right\rfloor - 1 \mbox{ (\textbf{Case B}).}
\end{cases}$$

As above, we define a difference of sums,

$$\msquare_m:=S_m(\{ a_1,a_2,a_3\},K)-S_m(\{ a_1+1, a_2, a_3\}, K-1),$$

and note that $\msquare_m$  can be converted to $\Delta_m$ via (\ref{eq:threetotwo}) as follows:

\begin{align*}
\msquare_m = & S_m(\{ a_1,a_2,a_3 \}, K) - S_m(\{ a_1+1, a_2, a_3 \}, K-1) \\
= &  ( S_m(\{ a_1, a_2+a_3 \}, K) - S_m(\{ a_1+1,a_2+a_3 \}, K-1) ) \\
& - ( S_m(\{ a_1, a_2 \}, K) - S_m(\{ a_1+1,a_2 \}, K-1) ) \\
& - ( S_m(\{ a_1, a_3 \}, K) - S_m(\{ a_1+1,a_3 \}, K-1) ) \\
= & \Delta_m(\{ a_1,a_2+a_3\}, K) - \Delta_m(\{ a_1,a_2 \}, K) - \Delta_m(\{ a_1, a_3 \}, K).
\end{align*}

\textbf{Case A:} From the proof of Theorem \ref{thm:2varlowerbound}, we know that $$\Delta_m(\{ a_1,a_2 \}, K)\geq 0 \mbox{ and }\Delta_m(\{ a_1, a_3 \}, K) \geq 0.$$ This, coupled with the fact that $\Delta_m(\{ a_1,a_2+a_3 \}, K) \leq 1$ (by Table \ref{table:chartdelta}) gives $\msquare_m \leq 1$. Furthermore, with $S_m(\{a_1,a_2,a_3\},0)\leq 1$, we get $S_m \leq K+1$ as in the proof of Theorem \ref{thm:2var}.

\textbf{Case B:} If $\msquare_m\leq 0$, then we can use Case A (i.e. $S_m\leq K+1 \leq \left\lfloor\frac{m}{3}\right\rfloor$) to show $$S_m(\{a_1,a_2,a_3\},K)\leq S_m(\{a_1+1,a_2,a_3\},K-1) \leq \left\lfloor\frac{m}{3}\right\rfloor.$$
If $\msquare_m=1$, we show $S_m\leq \left\lfloor\frac{m}{3}\right\rfloor$ directly.  Observe that the only way to obtain $\msquare_m=1$ is when

\begin{align*}
\Delta_m(\{ a_1,a_2+a_3 \}, K) & = +1,\\
\Delta_m(\{ a_1,a_2 \}, K) & = 0,\\
\Delta_m(\{ a_1,a_3 \}, K) & = 0.
\end{align*}
This arrangement is achieved when our assumed $a_1,a_2,a_3,K$ also satisfies all conditions from Table \ref{table:chartdelta}. So, in Table \ref{table:logical}, we organize each row according to the restrictions that must be applied. The `type' refers to the logical operator on the conditions in the same row in order to achieve that particular $\Delta_m$ value.
\begin{table}[ht]
{\renewcommand{\arraystretch}{1.2}
\begin{center}
\resizebox{\columnwidth}{!}{%
\begin{tabular}{|c|cc|c|cc|}
\hline
$\Delta_m$ && Condition \textbf{a} & Type && Condition \textbf{b} \\
\hline
+1 & (1a) & $a_1+\modd{(a_2+a_3)} {m} \geq m$ & AND & (1b) & $\modd{(a_2+a_3)} {m} + K - m + 1 \leq 0$\\
\hdashline
0 & (2a) & $a_1+a_2 \geq m$ & XOR & (2b) & $a_2 + K - m +1 \leq 0$\\
\hdashline
0 & (3a) & $a_1+a_3 \geq m$ & XOR & (3b) & $a_3 + K -m + 1 \leq 0$\\
\hline
\end{tabular}%
}
\end{center}
}
\caption{Conditions on $\Delta_m$ values for Case B}
\label{table:logical}
\end{table}

Keeping these conditions in mind, we bound $S_m(\{ a_1,a_2,a_3 \}, K)$ according to whether $a_2+a_3 < m$ or not.

\textbf{Case B1:} $a_2+a_3 < m.$ Conditions (1a) and (1b) simplify to $$m\leq a_1+a_2+a_3 < 2m \mbox{ AND } a_2+a_3+K-m+1 \leq 0,$$ which therefore satisfies (2b) and (3b), implying that it also satisfies $\sim\hspace{-1mm}$ (2a) and $\sim\hspace{-1mm}$ (3a) by the XOR condition.  Thus, we get

\begin{align*}
a_1 + a_2 < m & \mbox{ AND } a_2+K-m+1 \leq 0, \\
a_1 + a_3 < m & \mbox{ AND } a_3+K-m+1 \leq 0.
\end{align*}

Applying these conditions to Proposition \ref{prop:main3} results in the following formula:

\begin{align*}
S_m&(\{ a_1,a_2,a_3 \}, K)\\
=& K+1-\max (0,a_1+a_2+K-m+1)\\
&-\max (0,a_1+a_3+K-m+1)+\max (0,a_1+K-m+1),
\end{align*}

which can be broken down into four subcases:

\textbf{B1.1: ($a_1+K-m+1>0$)} 
$$ S_m(\{ a_1,a_2,a_3\}, K) = -a_1-a_2-a_3+m \leq -m + m =0 $$

\textbf{B1.2: ($a_1+K-m+1 \leq 0$ and $a_1+a_3+K-m+1 > 0$)}
\begin{align*}
S_m &(\{ a_1,a_2,a_3\}, K)\\
&= K+1 -(a_1+a_2+K-m+1)-(a_1+a_3+K-m+1)\\
&=(m-a_1-a_2-a_3)-(a_1+K)+(m-1)\\
&\leq 0 - \left\lfloor\frac{2m}{3}\right\rfloor + (m-1)\\
&\leq \left\lfloor\frac{m}{3}\right\rfloor
\end{align*}
because $a_1+a_2+a_3\geq m$ (1a) and $a_1\geq a_2 \geq a_3$ gives $a_1\geq \left\lfloor\frac{m}{3} \right\rfloor$.  This, along with $K \geq \left\lfloor\frac{m}{3}\right\rfloor$, gives us the second to last line.

\textbf{B1.3: ($a_1+a_3+K-m+1\leq 0$ and $a_1+a_2+K-m+1 > 0$)}
$$S_m(\{ a_1,a_2,a_3\}, K)=m-a_1-a_2 \leq \left\lfloor\frac{m}{3}\right\rfloor$$ because $a_1+a_2\geq \left\lfloor\frac{2m}{3}\right\rfloor$.

\textbf{B1.4: ($a_1+a_2+K-m+1\leq 0$}) This case cannot happen because $a_1+a_2+a_3\geq m$ gives that $a_1+a_2 \geq \left\lfloor\frac{2m}{3}\right\rfloor$ which means $a_1+a_2+K+1\geq m$, contradicting the condition.

We summarize these four subcases in Table \ref{table:thm13a}.

\begin{table}[ht]
{\renewcommand{\arraystretch}{1.2}
\begin{center}
\begin{tabular}{|c|c|c|c|c|}
\hline
Subcase & $a_1+a_2+K-m+1$ & $a_1+a_3+K-m+1$ &  $a_1+K-m+1$ &$S_m$\\
\hline
1 & $>0$ & $>0$ & $>0$ & $\leq 0$\\
2 & $>0$ & $>0$ & $\leq 0$ & $\leq \left\lfloor\frac{m}{3}\right\rfloor$\\
3 & $>0$ & $\leq 0$ & $\leq 0$ & $\leq \left\lfloor\frac{m}{3}\right\rfloor$\\
4 & $\leq 0$ & $\leq 0$ & $\leq 0$ & N/A\\
\hline
\end{tabular}
\end{center}
}
\caption{Subcases for B1}
\label{table:thm13a}
\end{table}

\textbf{Case B2:} $m\leq a_2+a_3 <2m.$

Conditions (1a) and (1b) simplify to $$a_1+a_2+a_3 \geq 2m \mbox{ AND } a_2+a_3+K-2m+1\leq 0,$$ which therefore satisfies (2a) and (3a), implying that it also satisfies\\ $\sim\hspace{-1mm}$ (2b) and $\sim\hspace{-1mm}$ (3b) by the XOR condition.  Thus, we get

\begin{align*}
a_1 + a_2 \geq m & \mbox{ AND } a_2+K-m+1 > 0 \\
a_1 + a_3 \geq m & \mbox{ AND } a_3+K-m+1 > 0.
\end{align*}

Applying these conditions to Proposition \ref{prop:main3} results in the following formula:

\begin{align*}
S_m& (\{ a_1,a_2,a_3 \}, K)= -(K+1)\\
&- \max (0,a_1+a_2+K-2m+1)-\max (0,a_1+a_3+K-2m+1)\\
&+(a_1+K-m+1)+(a_2+K-m+1)+(a_3+K-m+1),
\end{align*}

which can be broken down into three subcases:

\textbf{B2.1: ($a_1+a_2+K-2m+1>0$ and $a_1+a_3+K-2m+1>0$)} 

\begin{align*}
S_m& (\{ a_1,a_2,a_3\}, K)\\
=& -(a_1+a_2+K-2m+1)-(a_1+a_3+K-2m+1)\\
&+(a_1+a_2+K-2m+1)+(a_3+K-m+1)\\
=& m-a_1\\
\leq & \left\lfloor \frac{m}{3} \right\rfloor
\end{align*}

because $a_1+a_2+a_3\geq m$ and $a_1\geq a_2 \geq a_3$ gives $a_1 \geq \left\lfloor\frac{2m}{3} \right\rfloor$.

\textbf{B2.2: ($a_1+a_2+K-2m+1>0$ and $a_1+a_3+K-2m+1\leq 0$)}

\begin{align*}
S_m(\{ a_1,a_2,a_3\}, K) = & (m-a_1) + (a_1+a_3+K-2m+1)\\
\leq & m-a_1 \\
\leq & \left\lfloor\frac{m}{3}\right\rfloor
\end{align*}

\textbf{B2.3: ($a_1+a_2+K-2m+1\leq 0$)}

\begin{align*}
S_m(\{ a_1,a_2,a_3\}, K) = & (m-a_1)+(a_1+a_2+K-2m+1)\\
& + (a_1+a_3+K-2m+1)\\
\leq& (m-a_1)\\
\leq&\left\lfloor\frac{m}{3}\right\rfloor
\end{align*}

We summarize the three cases in Table \ref{table:thm13b}.

\begin{table}[ht]
{\renewcommand{\arraystretch}{1.2}
\begin{center}
\begin{tabular}{|c|c|c|c|}
\hline
Subcase & $a_1+a_2+K-2m+1$ & $a_1+a_3+K-2m+1$ & $S_m$\\
\hline
1 & $>0$ & $>0$ & $\leq \left\lfloor\frac{m}{3}\right\rfloor$\\
2 & $>0$ & $ \leq 0$ & $\leq \left\lfloor\frac{m}{3}\right\rfloor$\\
3 & $\leq 0$ & $\leq 0$ & $\leq \left\lfloor\frac{m}{3}\right\rfloor$\\
\hline
\end{tabular}
\end{center}
}
\caption{Subcases for B2}
\label{table:thm13b}
\end{table}

The results from Tables \ref{table:thm13a} and \ref{table:thm13b} completes the proof of Case B.  Hence, we have proven the theorem.

\end{proof}

Lemma \ref{thm:3reflect} and Theorem \ref{thm:3var} give the main result, which we state as a corollary.

\begin{corollary}
For $0\leq a_1,a_2,a_3, K\leq m-1$,	$$S_m(\{ a_1,a_2,a_3 \}, K) \leq \left\lfloor\frac{m}{3}\right\rfloor.$$
\label{thm:3varfull}
\end{corollary}

\begin{proof}
For $0\leq K \leq \left\lfloor \frac{m}{2}\right\rfloor -1$, Theorem \ref{thm:3var} gives us our result. For $\left\lfloor\frac{m}{2}\right\rfloor \leq K \leq m-2$, $S_m(\{ a_1,a_2,a_3\}, K) = S_m(\{ m-a_1,m-a_2,m-a_3 \}, m-2-K)$ by Lemma \ref{thm:3reflect}.  From there, we apply Theorem \ref{thm:3var} to the right hand side and get the result. Finally, $S_m(\{ a_1,a_2,a_3 \}, m-1)=0.$ This completes the proof.
\end{proof}

\section{(Not So Sharp) Lower Bound for \texorpdfstring{$n=4$}{}}

Given that $0 \leq a_1,a_2,a_3,a_4,K \leq m-1$, the pattern of the maximum and minimum values of $S_m(\{a_1,a_2,a_3,a_4\},K)$ is less clear, as evidenced by some results of the computer program:

\underline{Maximum Values of Sums}
\begin{align*}
(S_1,S_2, \ldots) = & (0,4,3,8,7,12,11,16,15,20,19,24 \\ & 23,28,27,32,31,36,35,40,39,44,\ldots).
\end{align*}

\underline{Minimum Values of Sums}
\begin{align*}
(S_1, S_2, \ldots) = & (0, 0, -3, -2, -3, -6, -5, -6, -9,\\
& -8, -9, -12, -11, -12, -15, -14, \\
& -15, -18, -17, -18, -21, -20,\ldots). 
\end{align*}
It has already been shown in \cite{day}, the upper bound $$S_m(\{a_1,a_2,a_3,a_4\},K) \leq 4\left\lfloor \dfrac{m}{2} \right\rfloor.$$
We conjecture the lower bound $$ -3\left\lfloor \dfrac{m}{3}\right\rfloor  \leq  S_m(\{a_1,a_2,a_3,a_4\},K).$$ In an attempt to prove this lower bound, we found that writing a difference of sums (like $\Delta_m$ or $\msquare_m$) is not an efficient way to approach the problem. Accordingly, we use another method to obtain the following partial result.

\begin{theorem}
For  $0 \leq a_1,a_2,a_3,a_4,K \leq m-1$, $$-2\left\lfloor \dfrac{m}{2} \right\rfloor-\left\lfloor \dfrac{m}{3} \right \rfloor\leq  S_m(\{a_1,a_2,a_3,a_4\},K) .$$
\end{theorem}

\begin{proof}

We can combine the following bounds from $n=2,3$, namely,
\begin{align*}
0 & \leq  S_m(\{a_1+a_2+a_3,a_4\},K), \\ 
-\left\lfloor \dfrac{m}{2} \right\rfloor & \leq  -S_m(\{a_1+a_2,a_4\},K), \\ 
-\left\lfloor \dfrac{m}{2} \right\rfloor & \leq  -S_m(\{a_1+a_3,a_4\},K), \\ 
-\left\lfloor \dfrac{m}{3} \right\rfloor & \leq  -S_m(\{a_2,a_3,a_4\},K), \\ 
0 & \leq  S_m(\{ a_1, a_4\}, K), 
\end{align*}

along with the identity

\begin{align*}
S_m(\{a_1,a_2,a_3,a_4\},K)&=S_m(\{a_1+a_2+a_3,a_4\},K)-S_m(\{a_1+a_2,a_4\},K) \\
&-S_m(\{a_1+a_3,a_4\},K)-S_m(\{a_2,a_3,a_4\},K)+ S_m(\{a_1,a_4\},K),
\end{align*}

to obtain the claimed result.

\end{proof}

\section{Conjectures}

In order to complete the analysis on this type of floor function problem, we want to show all the upper bounds and lower bounds for any number of variables, $n$. Onphaeng and Pongsriiam \cite{day} were able to show the upper bound when $n$ is even and the lower bound when $n$ is odd.

\begin{theorem}[Onphaeng, Pongsriiam] When $n$ is even and $m$ is even, $$S_m \leq  2^{n-2} \left \lfloor \dfrac{m}{2}\right\rfloor.$$ When $n$ is odd and $m$ is even, $$-2^{n-2}	 \left\lfloor \dfrac{m}{2} \right\rfloor \leq S_m.$$ The bounds on both cases are obtained exactly at $$A = \{m/2,m/2,\ldots,m/2\}, \;\ K=m/2-1.$$
\end{theorem}

We conjecture the missing bounds, namely, the lower bounds when $n$ is even and the upper bounds when $n$ is odd. To make these conjectures, we wrote a Maple program to calculate the values of $S_m(A,K)$ for specific $m, A,$ and $K$. For each $n$, the sets $A$ and $K$ that give extreme values of $S_m$ form interesting patterns, which depend on $m$. Once we determine such values in $A$ and $K$, we can quickly compute the extreme values of $S_m(A,K)$ for each $n$ and then use the resulting data to form a holonomic ansatz. The resulting recurrence becomes a ninth order recurrence with polynomial coefficients of degree at most 2. We summarize our findings in the following conjecture. For interested readers, this Maple code can be found on Thanatipanonda's website (\url{www.thotsaporn.com}).

\begin{conjecture} 
Let $A := \{a_1, a_2, \dots, a_n\}$ and define $\mbox{max/min }S_m(A,K)$ to be the maximum/minimum over all choices of $a_i$ and $K$ with  $0 \leq a_1, a_2, \dots, a_n, K \leq m-1$ and $m$ fixed.  Furthermore, we let the function $S_m$ be as defined in the first section. Suppose now

$$M(n):= \begin{cases} 
\mbox{max }S_m(A,K), & n \mbox{ odd}; \\
\mbox{min }S_m(A,K), & n \mbox{ even}.
\end{cases}$$ Then for $n \geq 4$, we conjecture the following result in two parts. %

\begin{enumerate}
\item{$n=4k-1, 4k$ or $4k+1,$ where $k\in\mathbb{Z}^+$
\vspace{0.2cm}

Under the condition that $m$ is a multiple of $2k+1$, the values of $M(n)$ occur exactly at
\begin{align*}
& A = \left\{\dfrac{km}{2k+1},\dfrac{km}{2k+1},\ldots,\dfrac{km}{2k+1}\right\} , K=\dfrac{km}{2k+1}-1 \\
\mbox{or }& A = \left\{\dfrac{(k+1)m}{2k+1},\dfrac{(k+1)m}{2k+1},\ldots,\dfrac{(k+1)m}{2k+1}\right\}, K=\dfrac{(k+1)m}{2k+1}-1.
\end{align*}}

\item{$n=4k+2,$ where $k\in\mathbb{Z}^+$
\vspace{0.2cm}

Under the condition that $m$ is a multiple of $2k+1$ and $2k+3$, the values of $M(n)$ occur (among other places) at
\begin{align*} 
& A = \left\{\dfrac{km}{2k+1},\dfrac{km}{2k+1},\ldots,\dfrac{km}{2k+1}\right\} , K=\dfrac{km}{2k+1}-1 \\ 
\mbox{or } & A = \left\{\dfrac{(k+1)m}{2k+1},\dfrac{(k+1)m}{2k+1},\ldots,\dfrac{(k+1)m}{2k+1}\right\}, K=\dfrac{(k+1)m}{2k+1}-1 \\
\mbox{or } & A = \left\{\dfrac{(k+1)m}{2k+3},\dfrac{(k+1)m}{2k+3},\ldots,\dfrac{(k+1)m}{2k+3}\right\}, K=\dfrac{(k+1)m}{2k+3}-1 \\
\mbox{or } & A = \left\{\dfrac{(k+2)m}{2k+3},\dfrac{(k+2)m}{2k+3},\ldots,\dfrac{(k+2)m}{2k+3}\right\}, K=\dfrac{(k+2)m}{2k+3}-1.
\end{align*}

Moreover, $M(n)$ can be calculated directly from a formula similar to the equations from Propositions \ref{prop:main2} and \ref{prop:main3}, or by $$M(n)= m\cdot f(n),$$ where $f(n)$ satisfies the recurrence relation 

\begin{align*}
-5(n+3)(n-2)f(n)=& 10(n^2+n-8)f(n-1)-4(2n^2-10n+3)f(n-2)\\
&-24(2n-11)f(n-3)-32(2n^2-10n-1)f(n-4)\\
&-192(n-1)(n-5)f(n-5)+64(2n^2-22n+51)f(n-6)\\
&+384(2n-13)f(n-7)-256(n-3)(n-8)f(n-8)\\
&+512(n-9)(n-8)f(n-9),
\end{align*}

for $n \geq 11$ with the initial conditions

\begin{align*}
& f(2)=0, f(3)=1/3, f(4)=-1, f(5)=2, f(6)=-3, f(7)=8, f(8)=-18,\\
& f(9)=36, f(10)=-65.
\end{align*}
}

\end{enumerate}
\label{conj:main}

\end{conjecture}

For convenience, we give examples of some of the bounds (and the set A for which the values of those bounds occur) produced from the conjectures above.
\begin{itemize}
\item{$n=4,5$, $m$ is a multiple of 3:
\begin{itemize}
\item[]{$n=4$: \[-3 \cdot \left\lfloor \dfrac{m}{3}\right\rfloor \leq S_m \]}
\item[]{$n=5$: \[S_m \leq 6 \cdot \left\lfloor\dfrac{m}{3}\right\rfloor \]}
\end{itemize}

For these cases, $M(n)$ occurs at:
\begin{align*}
& A = \{m/3,m/3,\ldots,m/3\}, K=m/3-1\\ 
\mbox{or }& A = \{2m/3,2m/3,\ldots,2m/3\} , K=2m/3-1.
\end{align*}
}

\item{$n=6$:
\begin{itemize} 
\item[$\bullet$]{$m$ is a multiple of 3:
\[ -9 \cdot \left\lfloor\dfrac{m}{3}\right\rfloor \leq S_m . \]
with the minimum at (among other places)

\begin{align*}
&A = \{m/3,m/3,\ldots,m/3\} , K=m/3-1 \\
\mbox{or } & A = \{2m/3,2m/3,\ldots,2m/3\} , K=2m/3-1.
\end{align*}
}

\item[$\bullet$]{$m$ is a multiple of 5:
\[ -15 \cdot \left\lfloor\dfrac{m}{5}\right\rfloor \leq S_m . \]
with the minimum is at (among other places) at

\begin{align*}
& A = \{2m/5,2m/5,\ldots,2m/5\} , K=2m/5-1\\
\mbox{or } & A = \{3m/5,3m/5,\ldots,3m/5\} , K=3m/5-1.
\end{align*}
}
\end{itemize}
}

\item{$n=7,8,9$, $m$ is a multiple of 5:
\begin{itemize}
\item[]{$n=7:$ \[  S_m \leq 40 \cdot \left\lfloor\dfrac{m}{5}\right\rfloor \]}
\item[]{$n=8:$ \[ -90 \cdot \left\lfloor \dfrac{m}{5}\right\rfloor \leq S_m \]}
\item[]{$n=9:$ \[  S_m \leq 180 \cdot \left\lfloor \dfrac{m}{5}\right\rfloor \]}
\end{itemize}

For these cases, $M(n)$ occurs at
\begin{align*}
& A = \{2m/5,2m/5,\ldots,2m/5\} , K=2m/5-1\\
\mbox{or } & A = \{3m/5,3m/5,\ldots,3m/5\} , K=3m/5-1.
\end{align*}
}

\end{itemize}

\section{Acknowledgements}

The authors would like to thank Harry Richman and an anonymous referee, for their helpful comments and suggestions to improve this manuscript.

\bigskip
\hrule
\bigskip

\noindent 2010 {\it Mathematics Subject Classification}:
Primary 11A25; Secondary 05D99

\noindent \emph{Keywords: } 
floor functions, optimization, sums

\end{document}